\documentclass[12pt]{amsart}

\newtheorem{theorem}{Theorem}[section]
\newtheorem{lem}[theorem]{Lemma}

\theoremstyle{definition}
\newtheorem{definition}[theorem]{Definition}

\theoremstyle{remark}

\def\tr{\mathop{\mathrm{tr}}\nolimits}

\begin{document}
\title[gradient pseudo-Ricci solitons]{Gradient pseudo-Ricci solitons of real hypersurfaces}
\author{MAYUKO KON}
\address{Faculty of Education, Shinshu University, 6-Ro, Nishinagano, Nagano City 380-8544, Japan}
\email{mayuko$\_$k@shinshu-u.ac.jp}

\maketitle
\begin{abstract}
Let $M$ be a real hypersurface of a complex space form $M^n(c)$, $c\neq 0$. Suppose that the structure vector field $\xi$ of $M$ is an eigen vector field of the Ricci tensor $S$, $S\xi=\beta\xi$, $\beta$ being a function. We study on $M$, a gradient pseudo-Ricci soliton as an extended concept of Ricci soliton, closely related to pseudo-Einstein real hypersurfaces. We show that a $3$-dimensional ruled real hypersurface of $M^2(c), c<0$ admits a non-trivial gradient pseudo-Ricci soliton.
\end{abstract}

\maketitle

\section[1]{Introduction}

Our object in this paper is to study a gradient pseudo-Ricci soliton on a real hypersurface of a non-flat complex space form under the assumption that the structure vector field of a real hypersurface is an eigen vector field of the Ricci tensor. The concept of gradient pseudo-Ricci solitons are closely related to pseudo-Einstein metric on real hypersurfaces.

A Riemannian manifold $M$ is Einstein if its Ricci tensor is a scalar multiple of the identity at each point. Cartan and Thomas \cite{Th} have shown that an Einstein hypersurface of Euclidean space is a hypersphere if its scalar curvature is positive, and Fialkow \cite{Fi} classified Einstein hypersurfaces in spaces of constant curvature, while Einstein complex hypersurfaces of complex space forms were classified by Smyth \cite{Sm}. 

On the other hand, it is well known that there are no Einstein real hypersurfaces of a non-flat complex space form. So we need the notion of pseudo-Einstein real hypersurfaces. Let $M$ be a real hypersurface of a Kaehlerian manifold $\tilde{M}$ and $(\phi, \xi, \eta, g)$ be an {\it almost contact metric structure} on $M$ (cf. Blair \cite{Bl}, Yano-Kon \cite{YK}). If the Ricci tensor $S$ of $M$ is of the form $SX = aX + b\eta(X)\xi$ for some constants $a$ and $b$, then $M$ is called a {\it pseudo-Einstein real hypersurface}. As for the definition of pseudo-Einstein real hypersurfaces, see also Kon \cite{MK3}. When $b = 0$, $M$ is Einstein. We remark that if $a$ and $b$ are functions, then they become constants. Any pseudo-Einstein real hypersurface has at most three constant principal curvatures and becomes a Hopf hypersurface. For pseudo-Einstein real hypersurfaces of a non-flat complex space form, see Cicel-Ryan \cite{CR}, Kon \cite{Ko} and Montiel \cite{Mo} . In case of dim$M$=3, see also Ivey-Ryan \cite{IR} and Kim-Ryan \cite{KR}.

On the other hand, the geometry of Ricci solitons has been the focus of attention of many mathematicians. In particular, it has become more important after Perelman \cite{Pe} applied Ricci solitons to solve the Poincar\'e conjecture. A Ricci soliton is said to be trivial if the potential vector field is zero or Killing, in which case the metric is Einstein. A Ricci soliton is called a gradient Ricci soliton if its potential field is the gradient of some smooth function on $M$. Later, Ricci soliton was also studied in the theory of various submanifolds of Riemannian manifolds (cf. Chen \cite{Ch}). For example, in \cite{CK}, Cho-Kimura proved that a compact Hopf hypersurface of a non-flat complex space form does not admit a Ricci soliton, and that a ruled real hypersurface of a non-flat complex space form does not admit a gradient Ricci soliton. 

For real hypersurfaces in complex space form, the condition pseudo-Einstein play an important role instead of Einstein. As a corresponding extended concept, we define the notion of pseudo-Ricci solitons and gradient pseudo-Ricci solitons. First we show the existence of non-trivial gradient pseudo-Ricci soliton on $3$-dimensional ruled real hypersurface of a complex hyperbolic space. For important results on ruled real hypersurfaces related to our results, refer to Berndt and D\'iaz-Ramos \cite{BD}, Lohnherr and Reckziegel \cite{LR}. 

For the study of real hypersurfaces in a complex space form, it is popular condition that the real hypersurface is Hopf, that is, the structure vector field $\xi$ is a principal vector of the shape operator. However, if we assume that a real hypersurface is Hopf, gradient Ricci soliton and gradient pseudo-Ricci soliton turn to be trivial. So, in our study, we use the condition $S\xi=\beta\xi$, $\beta$ being a function. We remark that any Hopf hypersurface satisfies this condition and ruled real hypersurfaces are non-Hopf examples which satisfy $S\xi=\beta\xi$ (see Kon \cite{MK}, \cite{MK2}).

Our main theorem is the following

\begin{theorem} Let $M$ be a real hypersurface of a non-flat complex space form $M^n(c)$. Suppose the Ricci tensor $S$ of $M$ satisfies $S\xi=\beta\xi$ for some function $\beta$. If $M$ admits a gradient pseudo-Ricci soliton, then $M$ is a pseudo-Einstein real hypersurface of $M^n(c)$ or a $3$-dimensional ruled real hypersurface of $M^2(c), c<0$.
\end{theorem}

Moreover, in view of the proof of Theorem 1.1, we see that there does not exist a real hypersurface of a non-flat complex space form $M^n(c)$ admitting a gradient Ricci soliton. Thus, we obtain

\begin{theorem}
Let $M$ be a compact real hypersurface of a non-flat complex space form $M^n(c)$ and suppose that the Ricci tensor $S$ of $M$ satisfies $S\xi=\beta\xi$ for some function $\beta$. Then $M$ does not admit a Ricci soliton.
\end{theorem}

\section[2]{Real hypersurfaces} 

Let $M^n(c)$ denote the complex projective space of complex dimension $n$ (real dimension $2n$) with constant holomorphic sectional curvature $4c$. We denote by $J$ the almost complex structure of $M^n(c)$. The Hermitian metric of $M^n(c)$ will be denoted by $G$.

Let $M$ be a real $(2n-1)$-dimensional hypersurface immersed in $M^n(c)$. Throughout this paper, we suppose that $M$ is connected. We denote by $g$ the Riemannian metric induced on $M$ from $G$. We take the unit normal vector field $N$ of $M$ in $M^n(c)$. For any vector field $X$ tangent to $M$, we define $\phi$, $\eta$ and $\xi$ by
$$JX=\phi X+\eta(X)N, \hspace{1cm} JN=-\xi,$$
where $\phi X$ is the tangential part of $JX$, $\phi$ is a tensor field of type (1,1), $\eta$ is a 1-form, and $\xi$ is the unit vector field on $M$. We call $\xi$ the {\textit{structure vector field}}. Then they satisfy
$$\phi^2X=-X + \eta(X)\xi,\quad \phi\xi = 0, \quad \eta(\phi X)=0$$
for any vector field $X$ tangent to $M$. Moreover, we have
\begin{eqnarray*}
& &g(\phi X,Y)+g(X,\phi Y) = 0, \quad \eta(X)=g(X,\xi),\\
& &g(\phi X,\phi Y)=g(X,Y)-\eta(X)\eta(Y). 
\end{eqnarray*}
Thus $(\phi,\xi,\eta,g)$ defines an almost contact metric structure on $M$.

We denote by $\tilde{\nabla}$ the operator of covariant differentiation in $M^n(c)$, and by $\nabla$ the one in $M$ determined by the induced metric. Then the {\it Gauss and Weingarten formulas} are given respectively by
$$\tilde{\nabla}_XY={\nabla}_XY+g(AX,Y)N, \hspace{1cm} \tilde{\nabla}_XN = -AX,$$
for any vector fields $X$ and $Y$ tangent to $M$. We call $A$ the {\it shape operator} of $M$. For the almost contact metric structure on $M$, we have
$${\nabla}_X\xi=\phi AX, \hspace{1cm} ({\nabla}_X\phi)Y=\eta(Y)AX-g(AX,Y)\xi.$$

We denote by $R$ the Riemannian curvature tensor field of $M$. Then the {\it equation of Gauss} is given by
\begin{eqnarray*}
R(X,Y)Z&=&c\{g(Y,Z)X - g(X,Z)Y + g(\phi Y,Z)\phi X\\
& & - g(\phi X,Z)\phi Y - 2g(\phi X,Y)\phi Z\} \\
& & + g(AY,Z)AX - g(AX,Z)AY,
\end{eqnarray*}
and the {\it equation of Codazzi} by
$$(\nabla_XA)Y-(\nabla_YA)X = c\{\eta(X)\phi Y - \eta(Y)\phi X - 2g(\phi X, Y)\xi\}.$$
From the equation of Gauss, the Ricci tensor $Ric$ of type (0,2) and $S$ of type (1,1) of $M$ are given by
\begin{eqnarray*}
{\it{Ric}}(X,Y)&=&g(SX,Y)\nonumber\\
&=&(2n+1)cg(X,Y)-3c\eta (X)\eta (Y) \\
& &\quad + {\rm tr}Ag(AX,Y) -g(AX,AY),\nonumber
\end{eqnarray*}
where ${\rm tr}A$ is the trace of $A$. The scalar curvature $Sc$ of $M$ is given by $Sc={\rm tr}S$ and
\begin{eqnarray*}
Sc=4(n^2 - 1)c+ ({\rm tr}A)^2 -{\rm tr}A^2.
\end{eqnarray*}

If the shape operator $A$ of $M$ satisfies$A\xi=\alpha \xi$ for some function $\alpha$, then $M$ is called a \textit{Hopf hypersurface}. By the equation of Codazzi, we have the following result (cf. \cite{Ma}).\\

\noindent{\textbf{Proposition A}.} \textit{Let $M$ be a Hopf hypersurface in $M^n(c)$, $n\geq 2$, If $X\perp \xi$ and $AX=\lambda X$, then $\alpha=g(A\xi,\xi)$ is constant and}
$$(2\lambda-\alpha)A\phi X=(\lambda\alpha +2c)\phi X.$$

From the equation of Codazzi, we have the following lemma (\cite{MK3}).

\begin{lem}
Let $M$ be a real hypersurface in a complex space form $M^n(c)$, $n\geq 3$, $c\neq 0$. If there exists an orthonormal frame $\{\xi, e_1,\cdots,e_{2n-2}\}$ on a sufficiently small neighborhood $\mathcal{N}$ of $x\in M$ such that the shape operator $A$ can be represented as
$$A\xi=\alpha\xi+he_1,\ \ Ae_1=a_1 e_1+h\xi,$$
$$ Ae_j=a_j e_j \ \ (j=2,\cdots, 2n-2),$$
then we have, for any $i,j\geq 2$, $i\neq j$,\\
\begin{eqnarray}
& &(a_j-a_k)g(\nabla_{e_i}e_j,e_k)-(a_i-a_k)g(\nabla_{e_j}e_i,e_k)=0,\\
& &(a_j-a_1)g(\nabla_{e_i}e_j,e_1)-(a_i-a_1)g(\nabla_{e_j}e_i,e_1)\\
& &\quad +h(a_i+a_j)g(\phi e_i,e_j) =0,\nonumber\\
& &\{2c-2a_ia_j+\alpha (a_i+a_j)\}g(\phi e_i,e_j)-hg(\nabla_{e_i}e_j,e_1)\\
& &\quad +hg(\nabla_{e_j}e_i,e_1) =0,\nonumber\\
& &(a_j-a_i)g(\nabla_{e_i}e_j,e_i)-(e_ja_i)=0,\\
& &(a_1-a_i)g(\nabla_{e_i}e_1,e_i)-(e_1a_i)=0,\\
& &(a_1-a_j)g(\nabla_{e_i}e_1,e_j)+(a_j-a_i)g(\nabla_{e_1}e_i,e_j)\\
& &\quad +a_ihg(\phi e_i,e_j) =0,\nonumber\\
& &\{2c-2a_1a_i+\alpha (a_i+a_1)\}g(\phi e_i,e_1)+hg(\nabla_{e_1}e_i,e_1)\\
& &\quad +(e_i h) =0,\nonumber\\
& &h(2a_i+a_1)g(\phi e_i,e_1)+(a_1-a_i)g(\nabla_{e_1}e_i,e_1)+(e_ia_1)=0,\\
& &h g(\nabla_{e_i}e_1,e_i)-(\xi a_i)=0,\\
& &(c+a_i\alpha -a_ia_j)g(\phi e_i,e_j)+h g(\nabla_{e_i}e_1,e_j)\\
& &\quad +(a_j-a_i)g(\nabla_{\xi}e_i,e_j)=0,\nonumber\\
& &(c+a_i\alpha -a_1a_i+h^2)g(\phi e_i,e_1)+(a_1-a_i)g(\nabla_{\xi}e_i,e_1)\\
& &\quad +(e_i h)=0,\nonumber\\
& &h(\alpha -3a_i)g(\phi e_i,e_1)+hg(\nabla_{\xi}e_i,e_1)+(e_i\alpha)=0,\\
& &(e_1 h)-(\xi a_1)=0,\\
& &(e_1 \alpha)-(\xi h)=0,\\
& &(c+a_1\alpha -a_1a_i-h^2)g(\phi e_1,e_i)-(a_1-a_i)g(\nabla_{\xi}e_1,e_i)\\
& &\quad +hg(\nabla_{e_1}e_1,e_i)=0,\nonumber
\end{eqnarray}
for any $i,j\geq 2$, $i\neq j$.
\end{lem}

We define the subspace $L_{x}\subset T_{x}(M)$ as the smallest subspace that contains $\xi$ and is invariant under the shape operator $A$. Then $M$ is Hopf if and only if $L_{x}$ is one-dimensional at each point $x$. We use the following lemma (see \cite{MK2}).

\begin{lem} Let $M$ be a real hypersurface of $M^n(c)$. If the Ricci tensor $S$ of $M$ satisfies $S\xi=\beta\xi$ for some function $\beta$, then ${\rm dim}L_{x} \leq 2$ at each point $x$ of $M^n(c)$.\end{lem}

\begin{lem} Let $M$ be a non-Hopf hypersurface of $M^n(c)$. If the Ricci tensor $S$ of $M$ satisfies $S\xi=\beta\xi$ for some function $\beta$, then 
\begin{eqnarray}
{\tr}A=\alpha + a_1, \hspace{1cm} a_2+\cdots+a_{2n-2}=0,
\end{eqnarray}
\begin{eqnarray}
Sc=4(n^2-1)c+2(\alpha a_1 - h^2)-\sum_{j=2}^{2n-2}a_j^2.
\end{eqnarray}

\end{lem}

\begin{proof}
From the assumptions, using Lemma 2.2, we can take an orthonormal frame, given in Lemma 2.1, $\{\xi,e_1,\cdots,e_{2n-2}\}$, locally, such that $A$ is of the form 
$$A\xi=\alpha\xi+he_1,\ \ Ae_1=a_1 e_1+h\xi,$$
$$ Ae_j=a_j e_j, \ \ j=2,\cdots, 2n-2.$$
Then, we obtain
\begin{eqnarray*}
S\xi&=&(2n-2)c\xi +(\tr A)(he_1 +\alpha\xi)-A(he_1 +\alpha\xi)\\
&=&(\tr A-\alpha-a_1)he_1+\{(2n-2)c+(\tr A)\alpha -h^2-\alpha^2\}\xi\\
&=&\beta \xi.
\end{eqnarray*}
So we see that ${\tr}A=\alpha + a_1$ and $a_2+\cdots+a_{2n-2}=0$. Moreover, the Ricci tensor $S$ can be represented as
$$S\xi=\beta\xi,\ \ Se_j=\lambda_j e_j, \ \ j=1,\cdots,2n-2,$$
where $\beta$ and $\lambda_i$ are given by
\begin{eqnarray}
& &\beta=(2n-2)c+(\alpha a_1-h^2),\nonumber\\
& &\lambda_1=(2n+1)c+(\alpha a_1-h^2),\\
& &\lambda_j=(2n+1)c+ \tr A\cdot a_j-a_j^2, \quad j=2,\cdots,2n-2\nonumber.
\end{eqnarray}
Thus the scalar curvature $Sc$ is given by (17).
\end{proof}

Here, we state some basic properties of pseudo-Einstein real hypersurfaces.\\

{\bf Remark.} Let $M$ be a pseudo-Einstein real hypersurface with ${\rm dim}M>3$. Then the Ricci tensor $S$ of $M$ satisfies $SX=aX+b\eta(X)\xi$. Hence, $S\xi=\beta\xi$. Suppose $h\ne 0$, that is, $M$ is not Hopf. Then, by (18), we have $b=-3c$. Therefore, for any vector $X$, we obtain
$$A^2 X-({\rm tr}A)AX+(a-(2n+1)c)X=0.$$
Since $M$ is not totally umbilical, we see that $M$ has two principal curvatures $\lambda$ and $\mu$. Then $\lambda+\mu={\rm tr}A$ and $\lambda\mu=a-(2n+1)c$. On the other hand, ${\rm tr}A=p\lambda+q\mu, p+q=2n-1$. Thus we have $(p-1)\lambda+(q-1)\mu=0$. From Lemma 2.3 of \cite{Ko}, we see ${\rm rank}A>1$. Hence $p, q>1$. Consequently, $\lambda$ and $\mu$ are constant.

Here, for the tangent space at $x$, we put
$$T_{x}(M)=L\oplus T_{\lambda}\oplus T_{\mu},$$
where $L$ is the subspace spanned by $\xi$ and $e_1$, $T_{\lambda}=\{X:AX=\lambda X, X\perp \xi, e_1\},  T_{\mu}=\{X:AX=\mu X, X\perp \xi, e_1\}$. By (16) and $h\neq 0$, we have ${\rm dim}T_{\lambda}, {\rm dim}T_{\mu}\geq 2$.

 Let $X$ and $Y$ be unit vectors of $L$ such that $AX=\lambda X$ and $AY=\mu Y$. We note $\eta(X)\ne0$ and $\eta(Y)\ne0$. For any $Z\in T_{\lambda}$, by the equation of Codazzi, 
$$0=g((\nabla_{X}A)Z,X)-g((\nabla_{Z}A)X,X) =-3c\eta(X)g(Z,\phi X).$$
This implies $g(Z,\phi X)=0$, and hence $\phi X\in T_{\mu}$. Similarly, we have $\phi Y\in T_{\lambda}$. On the other hand, $\phi X=\phi(\eta(X)\xi+g(X,e_1)e_1)=g(X,e_1)\phi e_1$ Since $g(X,e_1)\ne 0$, we have $\phi e_1\in T_{\mu}$. Similarly, $\phi Y=g(Y,e_1)\phi e_1$ and hence $\phi e_1\in T_{\lambda}$. This is a contradiction. 

If ${\rm dim}M=3$, by (16), $a_2=0$. From this and (18), we have $\alpha a_1-h^2=0$ and hence $(e_2 \alpha)a_1+\alpha(e_2 a_1)-2h(e_2 h)=0$. Using the equation of Codazzi, we have $c=0$ (see (7), (8), (11) and (12)) . This is a contradiction. Therefore, $h=0$ and $M$ is a Hopf hypersurface.\\

We introduce an important example of non-Hopf hypersurface. A real hypersurface of a complex space form $M^n(c) (c\ne 0, n\geq 2)$ is called a {\it ruled real hypersurface} if the holomorphic distribution $H(M)=\cup_{x\in M}\{X\in T_{x}(M) : X\perp \xi\}$, subbundle of the tangent bundle $T(M)$, of $M$ is integrable and each maximal integral manifold is locally congruent to a totally geodesic complex hypersurface $M^{n-1}(c)$ of $M^n(c) $.
It is known that every ruled real hypersurface is constructed in the following manner. 

Take a regular curve $\gamma$ in $M^n(c)$, defined on some interval, with tangent vector field $X$. At each point of $\gamma$ there is a unique totally geodesic complex hypersurface which is locally congruent to $M^{n-1}(c)$ of $M^n(c)$ cutting $\gamma$ so as to be orthogonal to $X$ and $JX$. The union of these hypersurfaces is called a {\textit{ruled real hypersurface}} (cf. \cite{Ki}, \cite{LR}, \cite{MAK}). 

A ruled real hypersurface $M$ is characterized by the shape operator by
$$A\xi=\alpha\xi+he_1, \ \ Ae_1=h\xi, \ \ AX=0 $$
for any $X$ orthogonal to $\xi$ and $e_1$, where $e_1$ is a unit vector field orthogonal to $\xi$. The Ricci tensor $S$ of a ruled real hypersurface $M$ satisfies $S\xi=\beta\xi$, where $\beta=(2n-2)c-h^2$.\\

\section[3]{Gradient pseudo-Ricci solitons} 

We now state the definitions of a pseudo-Ricci soliton and a gradient pseudo-Ricci soliton.

A vector field $V$ on a Riemannian manifold $M$ is said to define a {\it Ricci soliton} if it satisfies
$$\frac{1}{2}L_{V}g + Ric = \lambda g,$$
where $L_{V}g$ is the Lie-derivative of the metric tensor $g$ with respect to $V$, $Ric$ is the Ricci tensor of type (0,2) and $\lambda$ is a constant. 
We call the vector field $V$ the {\it potential vector field} of the Ricci soliton. A Ricci soliton $(M, g, V, \lambda)$ is called {\it shrinking}, {\it steady} or {\it expanding} according to $\lambda> 0; \lambda= 0; or \lambda< 0$, respectively. A Ricci soliton is said to be {\it trivial} if the potential vector field $V$ is zero or Killing, in which case the metric is Einstein.

A Ricci soliton $(M, g, V, \lambda)$ is called a {\it gradient Ricci soliton} if its potential field is the gradient of some smooth function $-f$ on $M$, which called the {\it potential function}:
$${\rm Hess}f = Ric - \lambda g.$$
A gradient Ricci soliton $(M, g, f, \lambda)$ is said to be {\it trivial} if its potential function $-f$ is a constant. Trivial gradient Ricci solitons are trivial Ricci solitons since $V= -Df$. It was proved in Perelman \cite{Pe} that if $(M, g, \xi, \lambda)$ is a compact Ricci soliton, the potential field is a gradient of some smooth function $f$ up to the addition of a Killing field. Thus compact Ricci solitons are gradient Ricci solitons.\\

In this article, we study real hypersurfaces of a non-flat complex space form. But, it is well known that there are no Einstein real hypersurfaces of a non-flat complex space form. Therefore we define

\begin{definition} A vector field $V$ on a real hypersurface $M$ of a non-flat complex space form with an almost contact metric structure $(\phi,\xi,\eta,g)$ is said to define a {\it pseudo-Ricci soliton} if it satisfies
$$\frac{1}{2}L_{V}g + Ric = \lambda g + \mu\eta\otimes\eta,$$
where $\lambda$ and $\mu$ are constants. 
\end{definition}

A trivial pseudo-Ricci soliton is one for which $V$ is zero or Killing, in which case the metric is pseudo-Einstein.\\

\begin{definition} A pseudo-Ricci soliton $(M, g, V, \eta, \lambda, \mu)$ is said to be {\it gradient} if its potential field is the gradient of a potential function $-f$ on $M$:
$${\rm Hess}f = Ric - \lambda g - \mu\eta\otimes\eta.$$
\end{definition}
A gradient pseudo-Ricci soliton $(M, g, f, \lambda,\mu)$ is called {\it trivial} if its potential function $-f$ is a constant. Then $M$ is a pseudo-Einstein real hypersurface. 

For a Riemannian manifold $M$ we generally have 
\begin{equation}\label{xx}
\begin{split}
&{\rm Hess}f(X,Y)={\rm Hess}f(Y,X),\\
&{\rm Hess}f(X,Y)=g(\nabla_{X}Df,Y)\\
&=\frac{1}{2}(g(\nabla_{X}Df,Y)+g(\nabla_{Y}Df,X))=\frac{1}{2}(L_{Df}g)(X,Y).
\end{split}
\end{equation}
Let $\{e_j\}$ be an orthonormal basis of $M$. Then the scalar curvature $Sc$ and the Ricci tensor $S$ satisfy
\begin{equation}
\nabla_X Sc=2\sum_j g((\nabla_{e_j}S)e_j,X).
\end{equation}

We now prepare the fundamental formulas for a gradient pseudo-Ricci soliton on a real hypersurface.

\begin{lem} Let $M$ be a real hypersurface of a complex space form $M^n(c), c\ne 0$. If $M$ admits a gradient pseudo-Ricci soliton, then
\begin{eqnarray}
g(\nabla_{X}Df,Y)=g(SX,Y)-\lambda g(X,Y)-\mu\eta(X)\eta(Y),
\end{eqnarray}
\begin{eqnarray}
g(R(X,Y)Df,Z)&=&g((\nabla_{X}S)Y,Z)-g((\nabla_{Y}S)X,Z)\\
& &-\mu g(\phi AX,Y)\eta(Z)-\mu\eta(Y)g(\phi AX,Z) \nonumber\\
& &+\mu g(\phi AY,X)\eta(Z)+\mu \eta(X)g(\phi AY,Z), \nonumber
\end{eqnarray}
\begin{eqnarray}
g(SX,Df)=-\frac{1}{2}\nabla_{X} Sc-\mu g(\phi A\xi,X).
\end{eqnarray}

\end{lem}

\begin{proof}
For any vector fields $X$ and $Y$, by the assumption and (19), we see
\begin{eqnarray*}
{\rm Hess}f(X,Y)&=&g(\nabla_{X}Df,Y)\\
&=&g(SX,Y)-\lambda g(X,Y)-\mu\eta(X)\eta(Y).
\end{eqnarray*}
This shows the first equation. Then, using the standard facts about covariant differentiation,
\begin{eqnarray*}
g(\nabla_X\nabla_Y Df,Z)&=&g(\nabla_X S)Y,Z)+g(S\nabla_X Y,Z)\\
& &\quad -\lambda g(\nabla_X Y,Z) - \mu g(\phi AX,Y)\eta(Z)\\
& &\quad -\mu\eta(\nabla_X Y)\eta(Z)-\mu\eta(Y) g(\phi AX,Z).
\end{eqnarray*}
We also have
\begin{eqnarray*}
g(\nabla_{[X,Y]}Df,Z)=g(S[X,Y],Z)-\lambda g([X,Y],Z)-\mu\eta([X,Y])\eta(Z).
\end{eqnarray*}
Using these equations and
\begin{eqnarray*}
& &g(R(X,Y)Df,Z)\\
& &=g(\nabla_X\nabla_Y Df,Z)-g(\nabla_Y\nabla_X Df,Z)-g(\nabla_{[X,Y]}Df,Z),
\end{eqnarray*}
we have equation (22). 

Let $\{e_j\}$ be an orthonormal basis of $M$. By (22), we obtain
\begin{eqnarray*}
& &\sum^{2n-1}_{j=1} g(R(X,e_j)Df,e_j) \\
& &=\sum_j g((\nabla_{X}S)e_j,e_j)-g((\nabla_{e_j}S)X,e_j)+\mu g(\phi A\xi,X).
\end{eqnarray*}
From the equation above and (20), we have (23).

\end{proof}

Using Lemma 3.3, we have the following lemma.

\begin{lem}
Let $M$ be a non-Hopf hypersurface in a complex space form $M^n(c)$, $c\neq 0$, $n\geq 3$. If the Ricci tensor $S$ of $M$ satisfies $S\xi=\beta\xi$ for some function $\beta$, and $M$ admits a gradient pseudo-Ricci soliton, then we have 
\begin{eqnarray}
& &-(e_1\beta)=c(e_1 f)+(a_1\alpha-h^2)(e_1f),\\
& &(\xi \lambda_1)=-c(\xi f) +(h^2 -a_1\alpha)(\xi f)\\
& &(\lambda_1-\lambda_j)g(\nabla_\xi e_1,e_j) + a_1(\lambda_j-\beta+\mu)g(\phi e_1,e_j)=0,\\
& &-(e_j\beta) +h(\beta-\lambda_j-\mu)g(\phi e_1,e_j)=c(e_j f)+a_j\alpha (e_jf),\\
& &a_j h(e_j f)=(\lambda_j-\lambda_1)g(\nabla_\xi e_j,e_1) + a_j(\lambda_1-\beta+\mu)g(\phi e_j,e_1),\\
& &(\xi \lambda_j)=-(c+a_j \alpha)(\xi f)- a_j h(e_1 f),\\
& &h a_j (e_j f)=\{a_1(\lambda_j -\beta)+a_j(\lambda_1-\beta)+ \mu(a_1+a_j)\}g(\phi e_j,e_1),\\
& &(\lambda_j - \lambda_1)g(\nabla_{e_1} e_j, e_1) - (e_j \lambda_1)\\
& &\quad =c\{(e_j f) -3(\phi e_1 f)g(\phi e_j,e_1)\} + a_1a_j(e_j f),\nonumber\\
& &(e_1\lambda_j) + (\lambda_j - \lambda_1)g(\nabla_{e_j} e_1,e_j)\\
& &\quad = c\{-(e_1 f)+3(\phi e_j f)g(\phi e_1,e_j)\}-a_1a_j(e_1f)-ha_j(\xi f)\nonumber\\
& &\{(a_i+ a_j)(\mu-\beta) + a_j\lambda_i + a_i\lambda_j\}g(\phi e_i,e_j)=0,\\
& &(\lambda_i-\lambda_1)g(\nabla_{e_j} e_i,e_1) + (\lambda_1 - \lambda_j)g(\nabla_{e_i} e_j,e_1)\\
& &\quad =c\{(\phi e_i f)g(\phi e_j, e_1)- (\phi e_j f)g(\phi e_i,e_1)\nonumber\\
& &\qquad -2g(\phi e_i,e_j)(\phi e_1 f)\},\nonumber\\
& &(\lambda_i - \lambda_j)g(\nabla_{e_j} e_i, e_j)-(e_i\lambda_j)\\
& &\quad =c\{(e_i f)-3(\phi e_j f)g(\phi e_i,e_j)\} + a_ia_j (e_i f),\nonumber\\
& &(\lambda_i-\lambda_k)g(\nabla_{e_j} e_i, e_k) -(\lambda_j - \lambda_k)g(\nabla_{e_i} e_j, e_k)\\
& &\quad =c\{(\phi e_i f)g(\phi e_j,e_k) - (\phi e_j f)g(\phi e_i, e_k)\nonumber\\
& &\qquad -2g(\phi e_j. e_i) g(\phi Df, e_k)\}.\nonumber
\end{eqnarray}
\end{lem}

\begin{proof}
From (22), we have
\begin{eqnarray*}
& &g(R(\xi, e_1)Df, \xi)\\
& &=g((\nabla_\xi S)e_1, \xi) - g((\nabla_{e_1} S)\xi, \xi)\\
& &=g(\nabla_\xi \lambda_1 e_1, \xi )- g(S\nabla_\xi e_1,\xi) - g(\nabla_{e_1}\beta\xi,\xi) + g(S\nabla_{e_1} \xi,\xi)\\
& &=-(e_1\beta).
\end{eqnarray*}
On the other hand, by the equation of Gauss,
\begin{eqnarray*}
g(R(\xi, e_1)Df, \xi)=c(e_1 f) + (a_1\alpha-h^2)(e_1 f).
\end{eqnarray*}
Thus we obtain (24). Similarly, substituting $e_1, e_j (j\geq 2), \xi$ into $X, Y, Z$ in (22), we have the other equations.
\end{proof}

Next we prepare the following
\begin{lem}
Let $M$ be a non-Hopf hypersurface in a complex space form $M^n(c)$, $c\neq 0$, $n\geq 3$. If the Ricci tensor $S$ of $M$ satisfies $S\xi=\beta\xi$ for some function $\beta$, and $M$ admits a gradient pseudo-Ricci soliton, then we have 
\begin{eqnarray}
& &\xi(\xi f)-h(\phi e_1 f)=\beta-\lambda-\mu,\\
& &e_1(\xi f)-a_1(\phi e_1 f)=0,\\
& &e_j(\xi f)-a_j(\phi e_j f)=0.
\end{eqnarray}
\end{lem}

\begin{proof}
By (21), we have
$$g(\nabla_\xi Df,\xi)=\beta-\lambda-\mu.$$
On the other hand, calculating the left-hand side implies
$$g(\nabla_\xi Df, \xi)=\nabla_\xi (\xi f) - g(Df, \phi A\xi)=\xi(\xi f)-h(\phi e_1 f).$$
So we have (37). Similarly, substituting $e_1, e_j (j\geq 2), \xi$ into $X, Y$ in (21), we have the other equations.
\end{proof}

At the end of this section, we have the following lemma directly from (23).

\begin{lem}
Let $M$ be a non-Hopf hypersurface in a complex space form $M^n(c)$, $c\neq 0$, $n\geq 3$. If the Ricci tensor $S$ of $M$ satisfies $S\xi=\beta\xi$ for some function $\beta$, and $M$ admits a gradient pseudo-Ricci soliton, then we have 
\begin{eqnarray}
& &\beta(\xi f)=-\frac{1}{2}\xi S_c,\\
& &\lambda_1(e_1 f)=-\frac{1}{2}e_1 S_c,\\
& &\lambda_j(e_j f)=-\frac{1}{2}e_j S_c -h\mu g(\phi e_1,e_j).
\end{eqnarray}
\end{lem}

\section[4]{Non-trivial example}

In order to describe the distinction of the paper easier, I will first show the existence of an important example of a 3-dimensional non-Hopf hypersurface. 

Let $M$ be a 3-dimensional real hypersurface of a complex space form $M^2(c), c\ne 0$ with $S\xi = \beta\xi$. Then $M$ is a Hopf hypersurface or a non-Hopf hypersurface such that
$$A\xi = \alpha \xi +he_1, \hspace{0.5cm}Ae_1=h\xi, \hspace{0.5cm}Ae_2=0\ \ (\phi e_1=e_2),$$
where $\{\xi, e_1, e_2=\phi e_1\}$ is an orthonormal basis.

In the following, we suppose that $M$ is a non-Hopf hypersurface, that is, $h\ne 0$. By (18), we have
\begin{eqnarray*}
& &\beta=g(S\xi,\xi)=2c+\alpha a_1-h^2,\\
& &\lambda_1=g(Se_1,e_1)=5c+\alpha a_1-h^2, \ \ \lambda_2=g(Se_2,e_2)=5c
\end{eqnarray*}
and the scalar curvature $Sc$ is given by
$$Sc=12c+\alpha a_1-h^2.$$

We remark that the equations in Lemma 2.1 hold except for the case that three distinct integers appear.

\begin{theorem} Let $M$ be a 3-dimensional non-Hopf hypersurface of a non-flat complex space form $M^2(c)$ with $S\xi = \beta\xi$. If $M$ admits a gradient pseudo-Ricci soliton, then $M$ is a ruled real hypersurface with $h^2=-c$, and
$$Df=\frac{1}{2}he_2,\hspace{0.5cm} \lambda=5c,\hspace{0.5cm} \mu=-\frac{5}{2}c.$$
\end{theorem}
\begin{proof} If $M$ admits a gradient pseudo-Ricci soliton, then
$$g(R(\xi,e_2)Df,e_2)=g((\nabla_{\xi}S)e_2,e_2)-g((\nabla_{e_2}S)\xi,e_2)=0.$$
From the equation of Gauss, using $Ae_2=0$, we see $\xi f=0$. Similarly,
$$g(R(e_1,e_2)Df,e_2)=g((\nabla_{e_1}S)e_2,e_2)-g((\nabla_{e_2}S)e_1,e_2)=0$$
reduce
$$-4cg(e_1,Df)=(\lambda_2-\lambda_1)g(\nabla_{e_2}e_1,e_2).$$
By (9) and $a_2=0$, $g(\nabla_{e_2}e_1,e_2)=0$, and hence $e_1 f=0$. If $e_2 f=0$, we conclude $Df=0$ and $M$ is a pseudo-Einstein real hypersurface. So we have $h=0$, which is a contradiction. Thus we can set $Df=me_2$, $m$ being a non-zero function. Then, from (21) we obtain
$$g(\nabla_{e_2}Df,e_2)=e_2 m = \lambda_2-\lambda=5c-\lambda.$$
We also have $g(\nabla_{e_1}Df,\xi)=-a_1 g(Df,e_2)=0$. Since $m=e_2f\neq 0$, we have $a_1=0$ and $M$ is a ruled real hypersurface.

By $g(\nabla_{\xi}Df,\xi)=-hg(Df,e_2)=-mh$ and (21), we obtain 
$$h^2-mh=2c-\lambda-\mu.$$
Hence $h^2-mh$ is a constant. So we have
$$2h(e_2 h)-(e_2m)h-m(e_2 h)=0.$$
On the other hand, by (11), $e_2 h = c+h^2$. Using $e_2 m=5c-\lambda$, 
$$2h^4-(mh)h^2+(-3c+\lambda)h^2-cmh=0.$$
Substituting $mh=-2c+h^2+\lambda+\mu$,
$$h^4-(2c+\mu)h^2+c(2c-\lambda-\mu)=0.$$
Consequently, we see $h$ is a constant, and hence $m$ is also a constant. Then we have $\lambda=5c$ and $c+h^2=0$, $c<0$. Moreover, by (21), we have $g(\nabla_{e_1}Df,e_1)=\lambda_1-\lambda=-h^2=c$. From $Df=me_2$, we see $mg(\nabla_{e_1}e_2,e_1)=c$. By (15), we also have $hg(\nabla_{e_1}e_2,e_1)=2c$. From these equations we obtain $m=\frac{1}{2}h$. From (23), we also have
$$g(Se_2,Df)=-\frac{1}{2}e_2 Sc-\mu h= -\mu h,$$
because $Sc$ is a constant. From these equations, we have $\mu=-\frac{5}{2}c$. These complete our assertion.\\
\end{proof}

In Lemma 1 and Lemma 2 in \cite{MAK}, Maeda, Adachi and Kim studied a shape operator of ruled real hypersurfaces. We consider the case $\nu^2=h^2=|c|$, $c<0$ in Lemma 2. We prove this example satisfies the condition $S\xi=\beta\xi$ for some function $\beta$ and admits a gradient pseudo-Ricci soliton.

\begin{lem}
Let $M$ be a ruled real hypersurface in $M^2(c)$, $c<0$. If $h^2=-c$, then the Ricci tensor satisfies $S\xi=\beta \xi$ for some function $\beta$ and $M$ admits a gradient pseudo-Ricci soliton for
$$\lambda=5c,\ \mu=-\frac{5}{2}c,\ f(\rho(s))=\frac{h}{2}s,$$
where $\rho$ is a geodesic with $\dot{\rho}(0)=e_2(\rho(0))$, which is the integral curve of $e_2$ through the point $\rho(0)$.
\end{lem}

\begin{proof}
Using $A\xi=he_1+\alpha\xi$, $Ae_1=h\xi$ and $Ae_2=0$,
$$S\xi=(2c-h^2)\xi=3c\xi.$$
Next we have 
\begin{eqnarray*}
g(Se_1,e_1)=5c-h^2=6c.
\end{eqnarray*}
On the other hand, we obtain
$${\rm{Hess}}f(e_1,e_1)=g(\nabla_{e_1} Df, e_1)=\frac{h}{2}g(\nabla_{e_1} e_2, e_1).$$
From (15), we have 
$$h g(\nabla_{e_1} e_2, e_1)=c-h^2=-2h^2.$$
Thus we obtain ${\rm{Hess}}(f)(e_1,e_1)=-h^2=c$. Hence we have
$${\rm{Hess}}f(e_1,e_1)=g(Se_1,e_1)-5c g(e_1,e_1)+\frac{5c}{2}\eta(e_1)\eta(e_1).$$
Similarly we obtain 
$${\rm{Hess}}f(X,Y)=g(SX,Y)-5c g(X,Y) +\frac{5c}{2} \eta (X)\eta(Y)$$
when $X$ and $Y$ are one of $e_1,e_2,\xi$, respectively.
\end{proof}

From these results, we have the following

\begin{theorem}
A real hypersurface of $M^2(c)$ admits a gradient pseudo-Ricci soliton and the Ricci tensor $S$ of $M$ satisfies $S\xi=\beta\xi$ for some function $\beta$ if and only if $M$ is one of the following
\begin{itemize}
\item[(1)] pseudo-Einstein real hypersurface,
\item[(2)] a ruled real hypersurface with a unit vector field $e_1$ orthogonal to $\xi$ and the shape operator satisfies
$$A\xi=\alpha\xi \pm \sqrt{|c|}e_1,\ Ae_1=\pm \sqrt{|c|}\xi,\ A\phi e_1=0$$
for some function $\alpha$.
\end{itemize}
\end{theorem}

\section[5]{Non-Hopf hypersurfaces} 

If the Ricci tensor $S$ of a real hypersurface $M$ of a non-flat complex space form satisfies $S\xi=\beta\xi$, by Lemma 2.2, $M$ is a Hopf hypersurface or a non-Hopf hypersurface which satisfies (16), (17) and (18).
In this section, under the assumption $S\xi=\beta\xi$ and $n\geq 3$, we study a gradient pseudo-Ricci soliton on a non-Hopf hypersurfaces $M$. The purpose of this section is to prove the following theorem.

\begin{theorem} Let $M$ be a real hypersurface of a non-flat complex space form $M^n(c)$, $n\geq 3$. Suppose the Ricci tensor $S$ of $M$ satisfies $S\xi=\beta\xi$ for some function $\beta$. If $M$ admits a gradient pseudo-Ricci soliton, then $M$ is a Hopf hypersurface.
\end{theorem}

The proof of this theorem will follow from a series of lemmas. In the following, we suppose that $M$ is a real hypersurface of $M^n(c), c\ne 0, n\geq 3$ with $S\xi=\beta \xi$, and $M$ admits a gradient pseudo-Ricci soliton
$${\rm Hess}f = Ric - \lambda g - \mu\eta\otimes\eta,$$
where $-f$ on $M$ denote the potential function.

We study the case that $M$ is non-Hopf, and work in an open set where $h\neq 0$. Here, in view of Lemma 2.2, we take an orthonormal basis
$$\{\xi,e_1,\cdots,e_{2n-2}\}$$
locally, such that
$$A\xi=\alpha\xi+he_1,\ \ Ae_1=a_1 e_1+h\xi,\ \ Ae_j=a_j e_j\ \ (j=2,\cdots,2n-2).$$

\begin{lem} $\xi f=0$.
\end{lem}

\begin{proof}

By (5) and (9), we have
$$(a_1-a_i)(\xi a_i)=h(e_1 a_i).$$
Thus we obtain
$$\frac{1}{2}\xi a_i^2=a_1(\xi a_i)-h(e_1 a_i).$$
Using (16), we have
$$\xi (\sum_{j=2}^{2n-2} a_i^2)=0.$$
On the other hand, (16), (18) and (29) imply
$$\xi \sum_{j=2}^{2n-2}a_j^2=(2n-3)c(\xi f)=0.$$
So we obtain $\xi f=0$.\\

\end{proof}

\begin{lem} If $a_1\neq 0$, then $a_1\alpha-h^2$ is constant.
\end{lem}

\begin{proof}
Using Lemma 5.2 and (37), we have
$$-h(\phi e_1 f)=(2n-2)c+a_1\alpha-h^2-\lambda - \mu.$$
By (38), we get 
\begin{equation}
a_1(\phi e_1 f)=0,
\end{equation}
Since $a_1\neq 0$, we have $\phi e_1 f=0$. Thus we obtain
$$(2n-2)c + a_1\alpha - h^2 -\lambda-\mu=0.$$
Since $\lambda$, $\mu$ and $c$ are constant, $a_1\alpha- h^2$ is also constant.
\end{proof}

\begin{lem} $a_1=0$.
\end{lem}

\begin{proof}
We suppose $a_1\neq 0$. From Lemma 5.2 and (38), we obtain $\phi e_1 f=0$. By Lemma 5.2 and (39), we have
\begin{equation}
a_j(\phi e_j f)=0.
\end{equation}
If $a_j\neq 0$ for all $j\geq 2$, then $\phi e_j f=0$. By Lemma 5.2, we have $Df=0$. In this case $M$ is pseudo-Einstein. This contradicts the assumption that $M$ is not Hopf. So we see that there exists some $j$ such that $a_j=0$.
By Lemma 5.3 and (24), we obtain
$$(c+a_1\alpha-h^2)(e_1 f)=0.$$
We suppose $a_j=0$. Then, from (32) and (35), we have
\begin{eqnarray}
& &(\lambda_1-\lambda_j)g(\nabla_{e_j} e_1,e_j) =c\{(e_1 f)-3(\phi e_j f)g(\phi e_1,e_j)\},\nonumber\\
& &(\lambda_i - \lambda_j) g(\nabla_{e_j} e_i, e_j) =c\{(e_i f)-3(\phi e_j f)g(\phi e_i, e_j)\}.
\end{eqnarray}
By (9), we see that $g(\nabla_{e_j} e_1,e_j)=0$. So if $a_j=0$, then we have
$$c(e_1 f)=3c (\phi e_j f)g(\phi e_1,e_j).$$
We denote $p$ the number of $a_j$ that satisfies $a_j=0$. We remark that if $a_j\neq 0$, then $\phi e_j f=0$ from (44). Then we obtain
\begin{eqnarray*}
pc(e_1 f)&=&\sum_{a_j=0}3c (\phi e_j f)g(\phi e_1,e_j)\\
&=& \sum_{j=1}^{2n-2} 3c (\phi e_j f)g(\phi e_1,e_j)\\
&=& -3c(e_1 f).
\end{eqnarray*}
From this equation, we obtain $(e_1 f)=0$.

Next, by (4), when $a_j=0$ and $i\neq j$, we have $a_ig(\nabla_{e_j} e_i, e_j)=0$. Since $\lambda_i-\lambda_j=(a_1+\alpha)a_i-a_i^2$, we see that
$$(\lambda_i- \lambda_j)g(\nabla_{e_j} e_i, e_j)=0.$$
From (45), we obtain
$$(e_i f)- 3(\phi e_j f)g(\phi e_i,e_j)=0.$$
Thus we have
\begin{eqnarray*}
pc(e_i f)&=&\sum_{a_j=0} 3(\phi e_j f)g(\phi e_i,e_j)\\
&=& \sum_{j=1}^{2n-2} 3(\phi e_j f)g(\phi e_i,e_j)\\
&=& -3c(e_i f).
\end{eqnarray*}
Here we used $\phi e_1 f=0$. From this equation, we obtain $e_i f=0$ for $i\neq j$. 

If $p\geq 2$, there exist $a_j=0$ and $a_k=0$. Then we have $e_i f=0$ when $i\neq j$ or $i\neq k$. Thus we see that $e_i f=0$ for any $i\geq 2$, and hence $Df=0$. In this case, $M$ is pseudo-Einstein and Hopf, this is a contradiction.

Next we consider the case that $p=1$. If $a_j=0$, then we have $Df=m e_j$ for some function $m$. Since $a_1\neq 0$ and $a_i\neq 0$ for any $i\neq j$, (43) and (44) imply that $\phi e_1 f=0$ and $\phi e_i f=0\ (i\neq j)$. Moreover, from Lemma 5.1, we have $\xi f=0$. Since $e_j$ is spanned by $\phi e_i$, $i\neq j$, so we see that $e_j f=0$. Hence we have $Df=0$ and this is a contradiction.

\end{proof}

From the proof of this lemma, we also have

\begin{lem} If there exist $j\geq 2$ such that $a_j=0$, then we have $e_1 f=0$.
\end{lem}

Next we prove the following

\begin{lem} There exist $j\geq 2$ such that $a_j=0$.
\end{lem}

\begin{proof}
We suppose $a_j\neq 0$ for any $j\geq 2$. By Lemma 5.2 and (39), we have $\phi e_j f=0$ for any $j\geq 2$ and $\xi f=0$. So we can represent $Df=m\phi e_1$ for some function $m$. 

By Lemma 5.4, we have $a_1=0$. So (37) implies that
\begin{equation}
-h(\phi e_1 f)=\{(2n-2)c-h^2\}-\lambda -\mu.
\end{equation}
Since $Df=m\phi e_1$, we obtain
\begin{eqnarray}
g(\nabla_{e_1} Df, e_1)&=& g(\nabla_{e_1} m\phi e_1,e_1)\nonumber\\
&=& -mg(\phi e_1,\nabla_{e_1} e_1)\nonumber\\
&=& -\sum_{k\geq 2} m g(\phi e_1,e_k)g(e_k,\nabla_{e_1} e_1).\nonumber
\end{eqnarray}
Since $a_1=0$ and $a_j\neq 0$ for $j\geq 2$, (8) implies
$$g(\nabla_{e_1} e_1, e_j)=-2h g(\phi e_j, e_1).$$
From these equations, we obtain
\begin{eqnarray*}
g(\nabla_{e_1} Df,e_1)&=&\sum_{k\geq 2} 2mh g(\phi e_1,e_k)g(\phi e_k,e_1)\\
&=&-2mh.
\end{eqnarray*}
On the other hand, by (21), we have
$$g(\nabla_{e_1} Df, e_1) = \lambda_1-\lambda.$$
From these equations and (18), we obtain
\begin{equation}
-2mh=(2n+1)c - h^2 - \lambda.
\end{equation}
By (46) and (47), we have $mh=-3c-\mu$, and hence $mh$ is constant. Moreover, again using (46) and (47), we have
$$h^2=(2n-5)c-\lambda-2\mu.$$
So $h$ is constant, from which we see that $m$ is also constant.

From (7) and $a_1=0$, we obtain
\begin{equation}
(2c+a_j\alpha +2h^2)g(\phi e_j, e_1)=0.
\end{equation}
If there exist $e_i$ and $e_j$ such that $g(\phi e_j,e_1)\neq 0$ and $g(\phi e_i,e_1)\neq 0$, then we have 
$$2c+a_j\alpha+2h^2 = 2c+a_i\alpha +2h^2=0,$$
and hence $(a_j-a_i)\alpha=0$. When $\alpha=0$, using (11), (12) and $c+h^2=0$, we have $g(\nabla_\xi e_i,e_1)=0$ and $g(e_1,\phi e_i)=0$ for any $i\geq 2$. This is a contradiction. So we see that $\alpha\neq 0$. Hence if $g(\phi e_j,e_1)\neq 0$ and $g(\phi e_i,e_1)\neq 0$, then $a_i=a_j$. Thus we can represent $A\phi e_1=a\phi e_1$ for some function $a$. Taking a suitable permutation, we can put $\phi e_1=e_2$ and $a_2=a$. By (48), we remark that
\begin{equation}
2c+a_2\alpha+2h^2=0.
\end{equation}
Thus we see that $a_2\alpha$ is constant. 

By (21), we have
$$g(\nabla_{e_2} Df,e_2)=(2n+1)c+a_2\alpha -a_2^2 - \lambda.$$
Since $m$ is constant, we obtain
$$g(\nabla_{e_2} Df,e_2)=g(\nabla_{e_2} me_2,e_2)=0.$$
From these equations, we have
\begin{equation}
a_2^2=(2n+1)c+a_2\alpha - \lambda.
\end{equation}
Since $a_2\alpha$ is constant, we see that $a_2$ is constant and $\alpha$ is also constant since $a_2\neq 0$.

By (11) and (12), we have
\begin{eqnarray*}
& &(c+a_2\alpha + h^2) g(\phi e_2,e_1) - a_2g(\nabla_\xi e_2,e_1)=0,\\
& &h(\alpha-3a_2)g(\phi e_2,e_1) + hg(\nabla_\xi e_2,e_1)=0.
\end{eqnarray*}
From these equations, we obtain
$$c+2a_2\alpha-3a_2^2 + h^2=0.$$
Combining this with (49), we have
\begin{equation}
c+h^2+a_2^2=0.
\end{equation}
From these equations, we see that $\alpha=2a_2$.

By (8), we have $g(\nabla_{e_1} e_2,e_1)=-2h$. Using (31), we obtain
$$(\lambda_2-\lambda_1)g(\nabla_{e_1}e_2,e_1)=4cm.$$
Hence we see that $2cm=-h(a_2\alpha - a_2 + h^2)$. Since $\alpha= 2a_2$, by (51), we have $h=2m$. By (47), we have $\lambda=(2n+1)c$. So (50) implies that $a_2=\alpha$, this is a contradiction. 

\end{proof}

From this lemma, there exists $a_j= 0$, $j\geq 2$. For this integer $j$, using (9), we have $g(\nabla_{e_j} e_1, e_j)=0$. Besides, by Lemma 5.5 and (32), we have
$$3c(\phi e_j f)g(\phi e_1,e_j)=0,$$
Hence we see that if $a_j=0$, then we have $\phi e_j f=0$ or $g(\phi e_1,e_j)=0$. When $g(\phi e_1,e_j)=0$, by (11), we have $e_j h=0$. In this case, (7) implies that $g(\nabla_{e_1} e_j, e_1)=0$. Thus from (31), we have $e_j f=0$. 

Taking a suitable permutation, we can take an orthonormal basis 
$\{\xi,e_1,e_2,\cdots, e_{2n-2}\}$ such that $a_j=0$ and $\phi e_j f=0$ for $j=2,\cdots,q$, $a_j=0$ and $\phi e_j f\neq 0$ for $j=q+1,\cdots, r$, $a_j\neq 0$ for $j=r+1,\cdots, 2n-2$. We put $H_{01}=\langle e_2,\cdots, e_q\rangle$, $H_{02}=\langle e_{q+1},\cdots, e_r\rangle$, $H_0=\langle e_2,\cdots, e_r\rangle$, $H_1=\langle e_{r+1},\cdots, e_{2n-2}\rangle$, respectively. That is, for $j, t\geq 2$, 
\begin{eqnarray*}
& &H_{01}={\rm Span}\{e_j : a_j=0, \phi e_j f=0\},\\
& &H_{02}={\rm Span}\{e_j : a_j=0, \phi e_j f\ne 0\},\\
& &H_1={\rm Span}\{e_t : a_t\ne 0, \phi e_t f =0\},\\
& &H_0=H_{01}\oplus H_{02}.
\end{eqnarray*}
If $e_j\in H_{02}$, then $g(\phi e_1,e_j)=0$. This implies that $\phi e_1 \in H_{01}\oplus H_1$.

We note that if $e_t\in H_1$, then (30) implies
\begin{equation}
h(e_t f)=(3c+\mu)g(\phi e_t,e_1).
\end{equation}
On the other hand, when $e_i\in H_0$ and $e_t\in H_1$, by (33),
\begin{equation}
(\mu+ 3c+ h^2)g(\phi e_i, e_t)=0.
\end{equation}

\begin{lem}
We have $\mu + 3c+h^2\neq 0$.
\end{lem}

\begin{proof}

We suppose $\mu + 3c+h^2=0$. We remark that $h$ is constant in this case. When $e_j\in H_0$, by (11),
$$(c+h^2)g(\phi e_j,e_1)=0.$$
First we consider the case $c+h^2\neq 0$. In this case we see that $\phi e_1\in H_1$. Using (7) and (8), we have
\begin{eqnarray*}
& &(2c+a_t\alpha)g(\phi e_t,e_1)+hg(\nabla_{e_1} e_t, e_1)=0,\\
& &2a_t hg(\phi e_t, e_1)-a_j g(\nabla_{e_1} e_t,e_1)=0.
\end{eqnarray*}
From these equations, using $a_t\neq 0$, we have 
$$(2h^2+2c+a_t\alpha)g(\phi e_t,e_1)=0.$$ 
Thus we see that if $g(\phi e_1,e_t)\neq 0$, then $a_t\alpha=-2h^2-2c\neq 0$. So we can put $a_t=m$ where $m\alpha=-2c-2h^2$. Since $\phi e_1\in H_1$, we can represent $\phi e_1=\sum_{e_k\in H_1} \mu_k e_k$. We remark that if $\mu_k\neq 0$, then $A e_k=me_k$. Hence we obtain
$$A\phi e_1=\sum_{e_k\in H_1}\mu_k Ae_k=m\left(\sum_{e_k\in H_1} \mu_ke_k\right)=m\phi e_1.$$
Using (11), (26) and $A\phi e_1=m\phi e_1$, we obtain
\begin{eqnarray*}
& &(c+h^2)g(\phi e_t, e_1)-a_t g(\nabla_\xi e_t, e_1)=0,\\
& &(h^2+2c-m^2)g(\nabla_\xi e_1,e_t)=0.
\end{eqnarray*}
From these equations, we have $(c+h^2)(h^2+2c-m^2)=0$. Since $h^2+c\neq 0$, we have $m^2=h^2+2c$, from which we see that $m$ is constant. By $m\alpha=-2c-2h^2$, $\alpha$ is also constant. Thus, putting $e_t=\phi e_1$, (12) implies that
$$g(\nabla_{\xi}e_t,e_1)=\alpha-3m.$$
Using (8), we have $g(\nabla_{e_1} e_t ,e_1)=-2h$. From (28) we obtain $e_t f=h$. By (31), we have
$$(m\alpha - m^2 + h^2)g(\nabla_{e_1} e_t,e_1)=4ch.$$
From these equations, we have $h^2+c=0$, this is a contradiction.

Next we consider the case $h^2+c=0$. Again using (7) and (8), we have
$$a_t\alpha g(\phi e_t,e_1)=0.$$
If there exists $e_t\in H_1$ such that $g(\phi e_t,e_1)\neq 0$, then $a_t\alpha=0$. Since $a_t\neq 0$, we have $\alpha=0$. From (11) and $c+h^2=0$, we obtain $g(\nabla_\xi e_j,e_1)=0$ for any $j\geq 2$. From this and (12), we have
$$-3a_j hg(e_1,\phi e_j)=0$$
for any $j\geq 2$. This contradicts to the assumption that there exists $e_t\in H_1$ such that $g(\phi e_t,e_1)\neq 0$. Hence this case does not occur. So we see that for any $e_t\in H_1$, we have $g(\phi e_t,e_1)=0$, that is, $\phi e_1\in H_0$. Taking a suitable permutation, we can put $e_2=\phi e_1$. By (27) and $c+h^2=0$, we have $e_2 f=0$. Thus (31) implies that $g(\nabla_{e_1} e_2, e_1)=0$. By (7), we have $cg(\phi e_2, e_1)=0$. This is a contradiction.
\end{proof}

From this lemma and (53), we see that $g(\phi e_i, e_t)=0$ for $e_i\in H_0$ and $e_t\in H_1$, that is, $\phi H_0\perp H_1$. 

When $e_i\in H_{02}$, then $\phi e_i f\ne 0$ and hence $g(\phi e_i,e_1)=0$. So, we have $\phi H_{02}\subset H_0$. Moreover, when $e_i, e_j \in H_{02}$, since $a_i=a_j=0$, $e_if=0$ and $\phi e_j f\neq 0$, (35) implies that $g(\phi e_i, e_j)=0$. Hence we see that $\phi H_{02}\subset H_{01}$.

When $e_i,e_j\in H_{01}$ and $i\neq j$, since $a_i=a_j=0$ and $\phi e_i f=\phi e_j f=0$, (35) implies that $e_i f=0$ and $e_j f=0$. Thus we see that if ${\rm{dim}}H_{01}\geq 2$, then we have $e_if=0$ for any $e_i\in H_{01}$. On the other hand, since $\phi H_{02}\subset H_{01}$ and $\phi e_j f\neq 0$ for $e_j\in H_{02}$, from which we see that ${\rm{dim}}H_{02}=0$.

Next we consider the case ${\rm{dim}}H_{01}=1$, that is, $H_{01}=\langle e_2 \rangle$. Since $\phi H_{02}\subset H_{01}$, we see that ${\rm{dim}}H_{02}\leq 1$. If ${\rm{dim}}H_{02}=1$, then we have $\phi H_{02}=H_{01}$. On the other hand, when ${\rm{dim}}H_{02}=0$, since $\phi H_1\perp H_0$, we have $\phi e_2=\pm e_1$. We can suppose $\phi e_1=e_2$.

From these consideration, we have the following lemma.

\begin{lem} We have the following three cases:\\
{\rm Case 1}: ${\rm dim}H_{01}\geq 2$, $H_0=H_{01}$, $e_j f=0, \phi e_j f=0\ \ (e_j\in H_{01})$,\\
{\rm Case 2}: ${\rm dim}H_{01}=1$, $\phi H_{02}=H_{01}, e_j f\ne 0\ \ (e_j\in H_{01}), \phi e_1 \in H_1$,\\
{\rm Case 3}: ${\rm dim}H_{01}=1$, ${\rm{dim}}H_{02}=0$, $\phi H_1=H_1$, $H_0$ is spanned by $\phi e_1$.\\

\end{lem}

Next we show the following

\begin{lem} {\rm Case 1} and {\rm Case 2} do not occur.
\end{lem}

\begin{proof}
From (35), we have
\begin{eqnarray*}
& &(\lambda_t-\lambda_i)g(\nabla_{e_i} e_t,e_i) - (e_t \lambda_i)\\
& &=c\{(e_t f)-e(\phi e_i f)g(\phi e_t,e_i)\}+a_ia_t(e_j f).
\end{eqnarray*}
If $a_i=0$ and $a_t\neq 0$, by (4), we have $a_jg(\nabla_{e_i} e_t, e_i)=0$. Since $\lambda_i=(2n+1)c$ and $\lambda_t=(2n+1)c + a_t {\rm{tr}}A -a_t^2$, we obtain
\begin{eqnarray*}
& &(\lambda_t-\lambda_i)g(\nabla_{e_i} e_t,e_i) - (e_j \lambda_i)\\
& &\quad =(a_t{\rm{tr}}A-a_t^2)g(\nabla_{e_i} e_t, e_i)=0.
\end{eqnarray*}
On the other hand, since $\phi H_1\perp H_0$, we have $g(\phi e_t,e_i)=0$. From these equations, we see that $e_t f=0$ for any $e_t\in H_1$.

When the shape operator satisfies the Case 1, by Lemma 5.4 and Lemma 5.8, we have $Df=0$ and $M$ is pseudo-Einstein. This contradict the assumption that $h\neq 0$. 

Next we consider the Case 2. We put $e_i\in H_{01}$, $e_j\in H_{02}$ and $\phi e_j=e_i$. Since $a_i=a_j=0$ and $\phi e_1\in H_1$, (10) and (34) imply
\begin{eqnarray*}
& &2cg(\phi e_i,e_j)=hg(\nabla_{e_i} e_j,e_1)-hg(\nabla_{e_j} e_i,e_1),\\
& &h^2 g(\nabla_{e_j} e_i,e_1)-h^2 g(\nabla_{e_i} e_j,e_1)\\
& &\quad =-2g(\phi e_i,e_j)(\phi e_1 f)c.
\end{eqnarray*}
From these equation, we have $\phi e_1 f=h\neq 0$. Since $\phi e_1\in H_1$, this is a contradiction.
\end{proof}

We consider the Case 3 and prove the following lemma.

\begin{lem} We have ${\rm{dim}}H_{01}=1$, ${\rm{dim}}H_{02}=0$, $\phi e_1=e_2\in H_{01}$ and $a_1=a_2=0, a_t \ne 0, t=3,\dots,2n-2$. Moreover, we have 
$$ h^2=-c, \ \ Df=\frac{h}{2}\phi e_1, \ \ \lambda=(2n+1)c, \ \ \mu=-\frac{5c}{2} $$.
\end{lem}

\begin{proof}
From the same proof of Lemma 5.9, we have $e_t f=0$ for any $e_t\in H_1$. So we obtain $Df=m\phi e_1=me_2$ for some function $m$. By (37), we obtain
\begin{equation}
-mh=\{(2n-2)c-h^2\}-\lambda-\mu.
\end{equation}
On the other hand, by (21),
$$mg(\nabla_{e_1} e_2,e_1)=\lambda_1-\lambda.$$
Using (15), we have
$$hg(\nabla_{e_1} e_1, e_2)=-c+h^2.$$
Combining these equations, we obtain
\begin{equation}
h\{(2n+1)c-h^2 -\lambda\}=m(c-h^2).
\end{equation}
From (54) and (55), we have
\begin{equation}
(2c+\mu) h^2 + c\{(2n-2)c-\lambda-\mu\}=0.
\end{equation}
First we consider the case that $2c+\mu=0$. Then we have $\lambda=2nc$. Moreover (54) implies that $mh=h^2$. Since $h\neq 0$, we have $m=h$. Since $Df=he_2$, we have $g(\nabla_{e_2} Df, e_2)=e_2 h$. Thus, from (21), we obtain
$$e_2 h= \lambda_2-\lambda=c.$$
Combining this with (7) and (15), we have $h^2=0$. This is a contradiction.

Next we suppose that $2c+\mu\neq 0$. From (54) and (56), we see that $h$ and $m$ are constant. Using (7) and (15), we have
\begin{eqnarray*}
& &hg(\nabla_{e_1} e_1,e_2)=-2c,\\
& &hg(\nabla_{e_1} e_1,e_2)=-c+h^2.
\end{eqnarray*}
So we get $h^2+c=0$. 

Since $m$ is constant, we see that $g(\nabla_{e_2} Df, e_2)=0$. So we have
$$0=\lambda_2 - \lambda=(2n+1)c-\lambda.$$
Thus we have $\lambda=(2n+1)c$. Using (55), we obtain $h=2m$. So (54) implies that $\mu=-\frac{5c}{2}$.

\end{proof}

\begin{lem} {\rm Case 3} does not occur.
\end{lem}

\begin{proof}
By (11), we have $g(\nabla_\xi e_s, e_1)=0$ when $e_s\in H_1$. Thus (12) implies that $e_s \alpha=0$. Next, by (4) and (35), we have
\begin{eqnarray}
& &(a_s - a_t) g(\nabla_{e_t} e_s, e_t) -(e_s a_t)=0,\\
& &(\lambda_s-\lambda_t) g(\nabla_{e_t} e_s, e_t) - (e_s \lambda_j)=0.
\end{eqnarray}
when $e_s,e_t\in H_1$ and $s\neq t$. Combining these equations, we have
\begin{eqnarray*}
0&=& (a_s \alpha - a_s^2 - a_t\alpha + a_t^2)g(\nabla_{e_t} e_s, e_t) - e_s (a_t \alpha -a_t^2)\\
&=& (\alpha - a_s - a_t)(a_s - a_t) g(\nabla_{e_t} e_s ,e_t) + (2a_t -\alpha)(e_s a_t)\\
&=&(\alpha - a_s - a_t)(e_s a_t) + (2a_t - \alpha)(e_s a_t)\\
&=& (a_t - a_s)(e_s a_t).
\end{eqnarray*}
Thus, when $a_s\neq a_t$, we have $(e_s a_t)=0 $. On the other hand, by (4), if $a_s=a_t$, then we have $(e_s a_t)=0$. So we see that if $e_s,e_t\in H_1$ and $s\neq t$, then $e_s a_t=0$. Since $\sum_{t\geq 2}a_t=0$, we have
$$0=e_s (\sum_{t\geq 2} a_t)=e_s a_s.$$
Next, by (5) and (32), we obtain
\begin{eqnarray*}
& &-a_t g(\nabla_{e_t} e_1, e_t) - (e_1 a_t)=0,\\
& &(e_1 \lambda_s) + (a_s \alpha - a_s^2 + h^2) g(\nabla_{e_s} e_1, e_s)=0.
\end{eqnarray*}
From these equations, we have 
\begin{eqnarray*}
0&=&e_1 (a_s \alpha - a_s^2) + (a_s \alpha - a_s^2 + h^2) g(\nabla_{e_s} e_1, e_s)\\
&=&(\alpha - 2a_s) (e_1 a_s) + (a_s \alpha - a_s^2 + h^2) g(\nabla_{e_s} e_1, e_s)\\
&=& (a_s^2 + h^2)g(\nabla_{e_s} e_1, e_s).
\end{eqnarray*}
Here we note that (14) implies that $(e_1\alpha)=0$ since $h$ is constant by $h^2=-c$ in Lemma 5.10. Since $a_s^2 + h^2>0$, we have $g(\nabla_{e_s} e_1, e_s)=0$ for any $e_s \in H_1$. By (5) and (9), we obtain $e_1 a_s=0$, $\xi a_s=0$.

Next we consider $e_2 a_s$. From (4) and (35), we have
\begin{eqnarray*}
& & a_sg(\nabla_{e_s} e_2, e_s) = -e_2 a_s,\\
& &(a_s^2 - a_s\alpha)g(\nabla_{e_s} e_2 ,e_s) - (e_2 \lambda_s)=\frac{hc}{2},
\end{eqnarray*}
from which we obtain
\begin{equation}
a_s (e_2 a_s) - a_s (e_2\alpha)=\frac{hc}{2}.
\end{equation}

By (21), we obtain
$$g(\nabla_{e_s} Df, e_s)=\lambda_s - (2n+1)c = a_s\alpha - a_s^2.$$
Since $Df= \frac{h}{2} e_2$, we have
$$\frac{h}{2} g(\nabla_{e_s} e_2, e_s) = a_s\alpha -a_s^2.$$
By (4), we have
$$a_i g(\nabla_{e_s} e_2, e_s)+(e_2 a_s)=0.$$
From these equations, we have
\begin{equation}
h(e_2 a_s) = 2a_i(a_s^2 - a_s\alpha).
\end{equation}
Using (59), we obtain
\begin{equation}
2a_s^4 - 2a_s^3\alpha - \frac{1}{2}h^2 c - a_s h(e_2\alpha)=0.
\end{equation}
Next we compute $e_2\alpha$. By (12) and (26), we have $g(\nabla_\xi e_1,e_2)=0$ and
\begin{equation}
(e_2\alpha) = h\alpha.
\end{equation}
From (61), using $h^2=-c$, we obtain
$$2a_s^4 - 2a_s^3\alpha - \frac{1}{2}h^2 c - a_s h^2 \alpha=0.$$
So we obtain
$$a_s\alpha = \frac{2a_s^4 + \frac{c^2}{2}}{2a_s^2-c}>0.$$
We remark that $c<0$ by $h^2+c=0$, so we have $a_s^2-c>0$. Since $\sum_{s=1}^{2n-2} a_s =0$, this is a contradiction.
\end{proof}

From these lemmas we have our Theorem 5.1.\\

\section[6]{Hopf hypersurfaces}

In this section we consider the case that $M$ is a Hopf hypersurface of a non-flat complex space form $M^n(c)$, $n\geq 2$. We remark that a Hopf hypersurface satisfies the condition $S\xi =\beta\xi$ for some function $\beta$.

We take an orthonormal basis $\{\xi, e_1, \dots, e_{2n-2}\}$ such that
$$A\xi = \alpha \xi, \ \ Ae_j=a_j e_j, \ \ j=1,\dots, 2n-2.$$
We notice that $\alpha$ is a constant. We also obtain
$$S\xi = \beta \xi, \ \ Se_j=\lambda_j e_j, \ \ j=1,\dots, 2n-2,$$
where
\begin{eqnarray*}
\beta&=&(2n-2)c+\alpha\sum a_j,\\
\lambda_j&=&(2n+1)c+a_j{\rm tr}A-a_j^2,\\
Sc&=&4(n^2-1)c+({\rm tr}A)^2-\alpha^2-\sum a_j^2.
\end{eqnarray*}
By the equation of Codazzi, we have $\xi a_j=0$ for all $j=1,\dots, 2n-2$.
So we see that $\xi \beta =0, \xi \lambda_j=0, \xi( {\rm tr}A)=0$ and $\xi Sc=0.$
Moreover, from Proposition A, we have
\begin{eqnarray}
(2a_j-\alpha)\bar a_j=\alpha a_j+2c,
\end{eqnarray}
where we set $\bar a_j =g(A\phi e_j,e_j)$.

We show the following

\begin{theorem} Let $M$ be a Hopf hypersurface of a non-flat complex space form $M^n(c)$, $n\geq 2$. If $M$ admits a gradient pseudo-Ricci soliton, then $M$ is a pseudo-Einstein real hypersurface.
\end{theorem}

\begin{proof}
By (22), we have
$$g(R(\xi,e_j)e_j,Df)=0.$$
From the equation of Gauss,
$$g(R(\xi,e_j)e_j,Df)=(c+\alpha a_j)g(\xi, Df).$$
Hence we have
$$(c+\alpha a_j)g(\xi, Df)=0.$$
If $c+\alpha a_j=0$ for all $j$, then $a_j=a_i$ for all $i, j$. We put $a=a_j, j=1,\dots, 2n-2$. Then (63) implies $(2a-\alpha)a=c$. Hence we have
$$2a^2=c+\alpha a =0.$$
Thus we obtain $a=0$ and $c=0$. This is a contradiction.

So there exists $j$ such that $c+\alpha a_j\ne 0$, from which we obtain $g(\xi, Df)=0$. This shows $\xi f=0$. Therefore, by (21),
$${\rm Hess}f(e_j,\xi)=-g(Df,\phi Ae_j)=-a_j g(Df,\phi e_j)=0.$$
If $a_j\ne0$, then $g(Df,\phi e_j)=0$. When $a_j=0$, (63) implies $-\alpha \bar a_j=2c$. Thus we see that
$$\bar a_j=-\frac{2c}{\alpha}$$
is constant. Moreover, the equation of Gauss and (22) imply
\begin{eqnarray*}
g(R(\phi e_j,e_j)Df, e_j) &=& -4cg(\phi e_j, Df)\\
&=& -\bar a_j({\rm tr}A-\bar a_j)g(\nabla_{e_j}\phi e_j,e_j).\nonumber
\end{eqnarray*}
From the equation of Codazzi, we also have
$$\bar a_j g(\nabla_{e_j}\phi e_j, e_j)=0.$$
Hence we have $4cg(\phi e_j,Df)=0$. Consequently for all $j$, we have
$$g(\phi e_j,Df)=0.$$
Thus, we have $Df=0$, which proves our assertion.

\end{proof}

From Theorem 4.1, Theorem 5.3 and Theorem 6.1, we have Theorem 1.1. Regarding a Ricci soliton, we will state the following
\begin{theorem}
Let $M$ be a real hypersurface of a non-flat complex space form $M^n(c)$ and suppose that the Ricci tensor $S$ of $M$ satisfies $S\xi=\beta\xi$ for some function $\beta$. Then $M$ does not admit a gradient Ricci soliton.
\end{theorem}
\begin{proof}
When $M$ admits a gradient Ricci soliton, we have $\mu=0$. Suppose $M$ is not Hopf. Then, in the proof of Theorem 5.10, we see $\mu=-\frac{5}{2}c$. This is a contradiction. Let $M$ be a Hopf hypersurface. Then, from Theorem 6.1 and $\mu=0$, $M$ is an Einstein real hypersurface. This is also a contradiction. Hence $M$ does not admit a gradient Ricci soliton.
\end{proof}
This implies Theorem 1.2. From Theorem 5.1 and Theorem 6.1, we also have the following result.

\begin{theorem} Let $M$ be a real hypersurface of a non-flat complex space form $M^n(c)$. Suppose $n\geq 3 $ and the Ricci tensor $S$ of $M$ satisfies $S\xi=\beta\xi$ for some function $\beta$. If $M$ admits a gradient pseudo-Ricci soliton, then $M$ is a pseudo-Einstein real hypersurface.
\end{theorem}


\end{document}